\documentclass[11pt,leqno]{article}
\usepackage[doc]{optional}
\usepackage{color}
\usepackage{float}
\usepackage{soul}
\usepackage{graphicx}
\definecolor{labelkey}{rgb}{0,0.08,0.45}
\definecolor{refkey}{rgb}{0,0.6,0.0}
\definecolor{Brown}{rgb}{0.45,0.0,0.05}
\definecolor{lime}{rgb}{0.00,0.8,0.0}
\definecolor{lblue}{rgb}{0.5,0.5,0.99}

\usepackage{mathpazo}

\usepackage{amsmath}

\usepackage{stmaryrd}
\usepackage{amssymb}
\usepackage{theorem}
\newcommand{\nnn}{\ensuremath{{n\in{\mathbb N}}}}
\newcommand{\thalb}{\ensuremath{\tfrac{1}{2}}}
\newcommand{\menge}[2]{\big\{{#1}~\big |~{#2}\big\}}

\newcommand{\To}{\ensuremath{\rightrightarrows}}

\newcommand{\fenv}[1]%
{\ensuremath{\,\overrightarrow{\operatorname{env}}_{#1}}}
\newcommand{\benv}[1]%
{\ensuremath{\,\overleftarrow{\operatorname{env}}_{#1}}}

\newcommand{\scal}[2]{\left\langle{#1},{#2}  \right\rangle}

\newcommand{\zeroun}{\ensuremath{\left]0,1\right[}}
\newcommand{\RR}{\ensuremath{\mathbb R}}

\newcommand{\dom}{\ensuremath{\operatorname{dom}}}

\newcommand{\gr}{\ensuremath{\operatorname{gr}}}

\newcommand{\inte}{\ensuremath{\operatorname{int}}}

\newcommand{\ran}{\ensuremath{\operatorname{ran}}}

\newcommand{\epi}{\ensuremath{\operatorname{epi}}}

\newcommand{\Fix}{\ensuremath{\operatorname{Fix}}}
\newcommand{\Id}{\ensuremath{\operatorname{Id}}}
\newcommand{\bId}{\ensuremath{\operatorname{\mathbf{Id}}}}

\newcommand{\pinf}{\ensuremath{+\infty}}
\newcommand{\bb}{\ensuremath{\mathbf{b}}}
\newcommand{\bx}{\ensuremath{\mathbf{x}}}
\newcommand{\bX}{\ensuremath{{\mathbf{X}}}}
\newcommand{\bY}{\ensuremath{{\mathbf{Y}}}}
\newcommand{\bR}{\ensuremath{{\mathbf{R}}}}
\newcommand{\bM}{\ensuremath{{\mathbf{M}}}}
\newcommand{\bT}{\ensuremath{{\mathbf{T}}}}
\newcommand{\bzero}{\ensuremath{{\boldsymbol{0}}}}

\newcommand{\bL}{\ensuremath{{\mathbf{L}}}}

\newcommand{\bD}{\ensuremath{{\boldsymbol{\Delta}}}}

\newcommand{\bQ}{\ensuremath{{\mathbf{Q}}}}
\newcommand{\bA}{\ensuremath{{\mathbf{A}}}}

\newcommand{\bc}{\ensuremath{\mathbf{c}}}
\newcommand{\by}{\ensuremath{\mathbf{y}}}
\newcommand{\bz}{\ensuremath{\mathbf{z}}}

{\begin{list}{}{%
\settowidth{\labelwidth}{\textrm{#1~}}%
\setlength{\leftmargin}{\labelwidth+\labelsep}}}
{\end{list}}
\newtheorem{theorem}{Theorem}[section]

\newtheorem{corollary}[theorem]{Corollary}
\newtheorem{proposition}[theorem]{Proposition}
\newtheorem{definition}[theorem]{Definition}
\theoremstyle{plain}{\theorembodyfont{\rmfamily}
}
\theoremstyle{plain}{\theorembodyfont{\rmfamily}
}
\theoremstyle{plain}{\theorembodyfont{\rmfamily}
}
\theoremstyle{plain}{\theorembodyfont{\rmfamily}
\newtheorem{example}[theorem]{Example}}
\newtheorem{fact}[theorem]{Fact}
\theoremstyle{plain}{\theorembodyfont{\rmfamily}
\newtheorem{remark}[theorem]{Remark}}

\newcommand{\boxedeqn}[1]{%
    \[\fbox{%
        \addtolength{\linewidth}{-2\fboxsep}%
        \addtolength{\linewidth}{-2\fboxrule}%
        \begin{minipage}{\linewidth}%
        \begin{equation}#1\\[+5mm]\end{equation}%
        \end{minipage}%
      }\]%
  }

\allowdisplaybreaks 

\begin{document}

\title{\textsc{
Compositions and convex combinations
of asymptotically regular firmly nonexpansive
mappings are also asymptotically regular}}

\author{
Heinz H.\ Bauschke\thanks{
Mathematics, University
of British Columbia,
Kelowna, B.C.\ V1V~1V7, Canada. E-mail:
\texttt{heinz.bauschke@ubc.ca}.},
Victoria Mart\'{\i}n-M\'{a}rquez\thanks{
Departamento de An\'alisis Matem\'atico, Universidad de Sevilla, 
Apdo.~1160, 41080--Sevilla, Spain. E-mail:
\texttt{victoriam@us.es}.},
Sarah M.\ Moffat\thanks{
Mathematics,
University of British Columbia,
Kelowna, B.C.\ V1V~1V7, Canada.
E-mail:  \texttt{sarah.moffat@ubc.ca}.},
~and Xianfu Wang\thanks{
Mathematics, University of
British Columbia,
Kelowna, B.C.\ V1V~1V7, Canada. E-mail:
\texttt{shawn.wang@ubc.ca}.}}

\date{December 21, 2011}
\maketitle

\vskip 8mm

\begin{abstract} \noindent
Because of Minty's classical correspondence between
firmly nonexpansive mappings and maximally monotone operators, 
the notion of a firmly nonexpansive mapping has proven to be of 
basic importance
in fixed point theory, monotone operator theory, and convex optimization.
In this note, we show that if finitely many firmly nonexpansive mappings
defined on a real Hilbert space are given and each of these mappings
is asymptotically regular, which is equivalent to saying that
they have or ``almost have'' fixed points,
then the same is true for their composition. 
This significantly generalizes the result by Bauschke from 2003 
for the case of projectors (nearest point mappings). 
The proof resides in a Hilbert product space and it relies upon
the Brezis-Haraux range approximation result. 
By working in a suitably scaled Hilbert product space,
we also establish the asymptotic regularity of convex combinations.
\end{abstract}

{\small
\noindent
{\bfseries 2010 Mathematics Subject Classification:}
{Primary 47H05, 47H09; Secondary 47H10, 90C25. 
}

\noindent {\bfseries Keywords:}
Asymptotic regularity, 
firmly nonexpansive mapping,
Hilbert space,
maximally monotone operator,
nonexpansive mapping,
resolvent,
strongly nonexpansive mapping. 
}

\section{Introduction and Standing Assumptions}

Throughout this paper,
\boxedeqn{
\text{$X$ is
a real Hilbert space with inner
product $\scal{\cdot}{\cdot}$ }
}
and induced norm $\|\cdot\|$.
We assume that 
\boxedeqn{
\text{$m\in\{2,3,4,\ldots\}$ and $I := \{1,2,\ldots,m\}$. }
}
Recall that an operator $T\colon X\to X$ is 
\emph{firmly nonexpansive} (see, e.g., \cite{BC2011}, \cite{GK}, and
\cite{GR} for further information)
if $(\forall x\in X)(\forall y\in X)$ 
$\|Tx-Ty\|^2\leq \scal{x-y}{Tx-Ty}$ and
that a set-valued operator $A\colon X\To X$ is 
\emph{maximally monotone} if it is \emph{monotone}, i.e., 
for all $(x,x^*)$ and $(y,y^*)$ in the graph of $A$, we have
$\scal{x-y}{x^*-y^*}\geq 0$ and if the graph of $A$ cannot be properly
enlarged without destroying monotonicity. 
(We shall write $\dom A = \menge{x\in X}{Ax\neq\varnothing}$
for the \emph{domain} of $A$, $\ran A = A(X) = \bigcup_{x\in X}Ax$ for
the \emph{range} of $A$, and $\gr A$ for the \emph{graph} of $A$.)
These notions are equivalent (see \cite{Minty} and \cite{EckBer}) 
in the sense that 
if $A$ is maximally monotone, then
its \emph{resolvent} $J_A := (\Id+A)^{-1}$ is
firmly nonexpansive, and if $T$ is firmly nonexpansive,
then $T^{-1}-\Id$ is maximally monotone. 
(Here and elsewhere, $\Id$ denotes the identity operator on $X$.)
The Minty parametrization (see \cite{Minty} and also
\cite[Remark~23.22(ii)]{BC2011}) 
states that if $A$ is maximally monotone, then 
\begin{equation}
\label{e:MintPar}
\gr A = \menge{(J_Ax,x-J_Ax)}{x\in X}.
\end{equation}
In optimization, one main problem is to find zeros of maximally monotone
operators --- these zeros may correspond to critical points or solutions to
optimization problems. In terms of resolvents, the corresponding problem is
that of finding fixed points. 
For background material in fixed point theory and monotone operator
theory, we refer the reader to 
\cite{BC2011},
\cite{BorVanBook}, 
\cite{Brezis},
\cite{BurIus},
\cite{GK},
\cite{GR}, 
\cite{Rock70},
\cite{Rock98},
\cite{Simons1},
\cite{Simons2},
\cite{Zalinescu},
\cite{Zeidler2a},
\cite{Zeidler2b},
and \cite{Zeidler1}.

\emph{The aim of this note is to provide approximate fixed point
results for 
compositions and convex combinations of finitely many firmly nonexpansive
operators}. 

The first main result (Theorem~\ref{t:Main}) substantially extends a result by
Bauschke \cite{B2003} on the compositions of projectors to the composition of firmly
nonexpansive mappings. 
The second main result (Theorem~\ref{t:haupt}) extends a result by
Bauschke, Moffat and Wang \cite{BMWnear} on the convex combination of firmly nonexpansive
operators from Euclidean to Hilbert space.

The remainder of this section provides the standing assumptions used
throughout the paper. 

Even though the main results are formulated in the given Hilbert space $X$, 
it will turn out
that the key space to work in is the product space
\boxedeqn{
X^m := \menge{\bx = (x_i)_{i\in I}}{(\forall i\in I)\; x_i\in X}. 
}
This product space contains an embedding of the original space $X$ via
the \emph{diagonal} subspace
\boxedeqn{
\bD := \menge{\bx= (x)_{i\in I}}{x\in X}.
}
We also assume that we are 
given $m$ firmly nonexpansive operators $T_1,\ldots,T_m$;
equivalently, $m$ resolvents of maximally monotone operators
$A_1,\ldots,A_m$: 
\boxedeqn{
(\forall i\in I)\quad T_i = J_{A_i} = (\Id+A_i)^{-1}
\;\;\text{is firmly nonexpansive.}
}
We now define various pertinent operators acting on $X^m$.
We start with the Cartesian product operators
\boxedeqn{
\bT\colon X^m\to X^m\colon (x_i)_{i\in I}\mapsto (T_ix_i)_{i\in I}
}
and 
\boxedeqn{
\bA\colon X^m\To X^m\colon (x_i)_{i\in I}\mapsto (A_ix_i)_{i\in I}. 
}
Denoting the identity on $X^m$ by $\bId$, we observe that
\begin{equation}
J_{\bA} = (\bId + \bA)^{-1} = T_1\times\cdots \times T_m = \bT. 
\end{equation}
Of central importance will be the
\emph{cyclic right-shift operator}
\boxedeqn{
\bR\colon X^m\to X^m\colon (x_1,x_2,\ldots,x_m)\mapsto
(x_m,x_1,\ldots,x_{m-1}) 
}
and for convenience we set 
\boxedeqn{
\bM = \bId-\bR. 
}
We also fix strictly positive \emph{convex coefficients} (or weights)
$(\lambda_i)_{i\in I}$, i.e.
\boxedeqn{
\label{e:weights}
(\forall i\in I)\quad \lambda_i\in\zeroun 
\;\;\text{and}\;\;\sum_{i\in I}\lambda_i=1. 
}
Let us make $X^m$ into the Hilbert product space
\boxedeqn{
\bX := X^m, \quad\text{with}\;\;
\scal{\bx}{\by} = \sum_{i\in I}\scal{x_i}{y_i}. 
}
The orthogonal complement of $\bD$
with respect to this standard inner product is known 
(see, e.g., \cite[Proposition~25.4(i)]{BC2011}) to be 
\begin{equation}
\bD^\perp = \menge{\bx=(x_i)_{i\in I}}{\sum_{i\in I}x_i = 0}.
\end{equation}
Finally, given a nonempty closed convex subset $C$ of $X$, the 
\emph{projector} (nearest point mapping) onto $C$ is denoted by $P_C$. 
It is well known to be firmly nonexpansive.

\section{Properties of the Operator $\bM$}

In this section, we collect several useful properties of the operator
$\bM$, including its Moore-Penrose inverse 
(see \cite{Groetsch} and e.g.\ \cite[Section~3.2]{BC2011} 
for further information.). To that end, the following
result---which is probably part of the folklore---will 
turn out to be useful.

\begin{proposition}
\label{p:pinv}
Let $Y$ be a real Hilbert space
and let $B$ be a continuous linear operator from $X$ to $Y$
with adjoint $B^*$ and 
such that $\ran B$ is closed. 
Then the Moore-Penrose inverse of $B$ satisfies
\begin{equation}
B^\dagger = P_{\ran B^*}\circ B^{-1}\circ P_{\ran B}.
\end{equation}
\end{proposition}
\begin{proof}
Take $y\in Y$. Define the corresponding set of least squares solutions
(see, e.g., \cite[Proposition~3.25]{BC2011}) 
by $C := B^{-1}(P_{\ran B}y)$. Since $\ran B$ is closed,
so is $\ran B^*$ (see, e.g., \cite[Corollary~15.34]{BC2011});
hence\footnote{$\ker B = B^{-1}0 = \menge{x\in X}{Bx=0}$ denotes the
\emph{kernel} (or nullspace) of $B$.}, 
$U := (\ker B)^\perp = \overline{\ran B^*} = \ran B^*$. 
Thus, $C = B^\dagger y + \ker B =  B^\dagger y + U^\perp$. 
Therefore, 
since $\ran B^\dagger = \ran B^*$ (see, e.g.,
\cite[Proposition~3.28(v)]{BC2011}), 
$P_U(C) = P_U B^\dagger y = B^\dagger y$, as claimed.
\end{proof}

Before we present various useful properties of $\bM$, let us recall the
notion of a \emph{rectangular} (which is also known
as star or 3* monotone, see \cite{BH}) operator.
A monotone operator $B\colon X\To X$ is \emph{rectangular} if 
$(\forall(x,y^*)\in\dom B\times\ran B)$
$\sup_{(z,z^*)\in\gr B} \scal{x-z}{z^*-y^*}<\pinf$.

\begin{theorem}
\label{t:M}
Define\,\footnote{Here and elsewhere we write $S^n$ for the
$n$-fold composition of an operator $S$.}
\begin{equation}
\label{e:bL}
\bL\colon \bD^\perp \to \bX 
\colon \by \mapsto  \sum_{i=1}^{m-1}\frac{m-i}{m}\bR^{i-1}\by.
\end{equation}
Then the  following hold. 
\begin{enumerate}
\item
\label{t:M1}
$\bM$ is continuous, linear, and maximally monotone with 
$\dom\bM = \bX$.
\item
\label{t:M2}
$\bM$ is rectangular. 
\item 
\label{t:M3}
$\ker\bM = \ker \bM^* = \bD$. 
\item 
\label{t:M4}
$\ran\bM = \ran\bM^* = \bD^\perp$ is closed. 
\item 
\label{t:M4+}
$\ran\bL = \bD^\perp$. 
\item
\label{t:M5}
$\bM\circ \bL = \bId|_{\bD^\perp}$. 
\item 
\label{t:M6}
$\displaystyle \bM^{-1}\colon \bX\To\bX\colon
\by\mapsto 
\begin{cases}
\bL\by + \bD, &\text{if $\by\in\bD^\perp$;}\\
\varnothing, &\text{otherwise.}
\end{cases} 
$
\item
\label{t:M7}
$\bM^\dagger = P_{\bD^\perp} \circ \bL \circ P_{\bD^\perp}=
\bL\circ P_{\bD^\perp}$. 
\item
\label{t:M7+}
$\displaystyle
\bM^\dagger = \sum_{k=1}^{m}\frac{m-(2k-1)}{2m}\bR^{k-1}.$
\end{enumerate}
\end{theorem}
\begin{proof}
\ref{t:M1}: 
Clearly, $\dom\bM=\bX$ and $(\forall \bx\in\bX)$ $\|\bR\bx\| = \|\bx\|$.
Thus, $\bR$ is nonexpansive 
and therefore $\bM =\bId-\bR$ is maximally monotone
(see, e.g., \cite[Example~20.27]{BC2011}). 

\ref{t:M2}: 
See \cite[Example~24.14]{BC2011} 
and \cite[Step~3 in the proof of Theorem~3.1]{B2003}
for two different proofs of the rectangularity of $\bM$. 

\ref{t:M3}: 
The definitions of $\bM$ and $\bR$ and the fact
that $\bR^*$ is the cyclic left shift operator readily imply
that $\ker\bM = \ker\bM^* = \bD$. 

\ref{t:M4}, \ref{t:M5}, and \ref{t:M6}: 
Let $\by = (y_1,\ldots,y_m)\in \bX$. 
Assume first that $\by \in\ran\bM$. 
Then there exists $\bx=(x_1,\ldots,x_m)$ such that
$y_1 = x_1-x_m$, $y_2=x_2-x_1$, \ldots, and 
$y_m = x_m - x_{m-1}$. 
It follows that $\sum_{i\in I} y_i = 0$, i.e.,
$\by\in \Delta^\bot$ by \cite[Proposition~25.4(i)]{BC2011}. 
Thus,
\begin{equation}
\label{e:1124a}
\ran\bM \subseteq \bD^\perp.
\end{equation}
Conversely, assume now that $\by\in\bD^\perp$. 
Now set
\begin{equation}
\bx := \bL\by = \sum_{i=1}^{m-1}\frac{m-i}{m}\bR^{i-1}\by.
\end{equation}
It will be notationally convenient to wrap indices around
i.e., $y_{m+1} = y_1$, $y_0=y_m$ and likewise.
We then get 
\begin{equation}
(\forall i\in I)\quad
x_i = \frac{m-1}{m}y_i + \frac{m-2}{m}y_{i-1}
+ \cdots + \frac{1}{m}y_{i+2}.
\end{equation}
Therefore,
\begin{equation}
\sum_{i\in I}x_i =
\frac{m-1}{m}\sum_{i\in I}y_i 
+ 
\frac{m-2}{m}\sum_{i\in I}y_i 
+ 
\frac{1}{m}\sum_{i\in I}y_i 
= \frac{m-1}{2}\sum_{i\in I}y_i = 0.
\end{equation}
Thus $\bx\in\bD^\perp$ and 
\begin{equation}
\label{e:poppadoms}
\ran\bL \subseteq \Delta^\perp.
\end{equation}
Furthermore,
\begin{subequations}
\begin{align}
(\forall i\in I)\quad
x_{i}-x_{i-1} &= \frac{m-1}{m}y_i - \frac{1}{m}y_{i-1}
-\frac{1}{m}y_{i-2}-\cdots - \frac{1}{m}y_{i+1}\\
&= y_i - \frac{1}{m}\sum_{j\in I} y_j = y_i.
\end{align}
\end{subequations}
Hence $\bM\bx = \bx-\bR\bx = \by$ and thus $\by\in\ran\bM$. 
Moreover, in view of \ref{t:M3}, 
\begin{equation}
\bM^{-1}\by = \bx + \ker\bM = \bx + \bD.
\end{equation}
We thus have shown
\begin{equation}
\label{e:1124b}
\bD^\perp \subseteq \ran \bM.
\end{equation}
Combining \eqref{e:1124a} and \eqref{e:1124b}, we obtain 
$\ran\bM = \bD^\perp$. 
We thus have verified \ref{t:M5}, and \ref{t:M6}. 
Since $\ran\bM$ is closed, so is 
$\ran\bM^*$ (by, e.g., \cite[Corollary~15.34]{BC2011}). 
Thus \ref{t:M4} holds.

\ref{t:M7}\&\ref{t:M4+}: 
We have seen in Proposition~\ref{p:pinv} that
\begin{equation}
\bM^\dagger = P_{\ran \bM^*} \circ \bM^{-1}\circ P_{\ran \bM}.
\end{equation}
Now let $\bz\in\bX$. 
Then, by \ref{t:M4}, $\by := P_{\ran \bM}\bz = P_{\bD^\perp}\bz \in
\bD^\perp$.
By \ref{t:M6}, $\bM^{-1}\by = \bL\by + \bD$. 
So $\bM^\dagger\bz = P_{\ran \bM^*}\bM^{-1}P_{\ran\bM}\bz = 
P_{\ran \bM^*}\bM^{-1}\by =
P_{\bD^\perp}(\bL\by + \bD) =
P_{\bD^\perp}\bL\by = \bL\by =(\bL\circ P_{\bD^\perp})\bz$ 
because $\ran\bL \subseteq \bD^\perp$ by \eqref{e:poppadoms}. 
Hence \ref{t:M7} holds. 
Furthermore, by \ref{t:M4} and e.g.~\cite[Proposition~3.28(v)]{BC2011}, 
$\ran\bL = \ran\bL\circ P_{\bD^\perp} = \ran \bM^\dagger = \ran\bM^* =
\bD^\perp$ and so \ref{t:M4+} holds. 

\ref{t:M7+}: 
Note that $P_{\bD^\perp} = \bId - P_{\bD}$
and that $P_{\bD} = m^{-1}\sum_{j\in I} \bR^j$. 
Hence 
\begin{equation}
P_{\bD^\perp} = \bId - \frac{1}{m}\sum_{j\in I}\bR^j.
\end{equation}
Thus, by \ref{t:M7} and \eqref{e:bL}, 
\begin{align}
\bM^\dagger &= \bL\circ P_{\bD^\perp} = 
\frac{1}{m}\sum_{i=1}^{m-1}(m-i)\bR^{i-1} \circ 
\Big(\bId- \frac{1}{m}\sum_{j\in I}\bR^j\Big)\\
&= \frac{1}{m}\sum_{i=1}^{m-1}(m-i)\bR^{i-1} -
\frac{1}{m^2}\sum_{i=1}^{m-1}(m-i)\sum_{j\in I}\bR^{i+j-1}. 
\end{align}
Re-arranging this expression in terms of powers of $\bR$ and
simplifying leads to
\begin{equation}
\bM^\dagger = (\bId-\bR)^\dagger =
\sum_{k=1}^{m}\frac{m-(2k-1)}{2m}\bR^{k-1}.
\end{equation}
\end{proof}

\begin{remark}
Suppose that $\widetilde{\bL}\colon\bD^\perp\to\bX$
satisfies $\bM\circ\widetilde{\bL} = \bId|_{\bD^\perp}$.
Then 
\begin{equation}
\bM^{-1}\colon \bX\To\bX\colon
\by\mapsto 
\begin{cases}
\widetilde{\bL}\by + \bD, &\text{if $\by\in\bD^\perp$;}\\
\varnothing, &\text{otherwise.}
\end{cases} 
\end{equation}
One may show that
$\bM^\dagger = P_{\bD^\perp}\circ \widetilde{\bL}\circ  P_{\bD^\perp}$
and that $P_{\bD^\perp}\circ\widetilde{\bL}=\bL$  (see \eqref{e:bL}).
Concrete choices for $\widetilde{\bL}$ and $\bL$ are
\begin{equation}
\bD^\perp \to \bX 
\colon (y_1,y_2,\ldots,y_m)\mapsto
(y_1,y_1+y_2,\ldots,y_1+y_2+y_3+\cdots+y_m);
\end{equation}
however, the range of the latter operator is not equal
$\bD^\perp$ whenever $X\neq\{0\}$. 
\end{remark}

\begin{remark}
Denoting the \emph{symmetric part} of 
$\bM$ by $\bM_+ = \thalb\bM+\thalb\bM^*$ and defining the 
\emph{quadratic form} associated with $\bM$ by 
$q_{\bM}\colon \bx \to \thalb\scal{\bx}{\bM\bx}$,
we note that 
\cite[Proposition~2.3]{B2003} implies that\footnote{
Recall that the \emph{Fenchel conjugate} 
of a function $f$ defined on $X$ is given by
$f^*\colon x^*\mapsto \sup_{x\in X}\big(\scal{x}{x^*}-f(x)\big)$.}
$\ran \bM_+ = \dom q_{\bM}^* = \bD^\perp$. 
\end{remark}

\begin{fact}[Brezis-Haraux]
\label{f:B-H} 
{\rm (See \cite{BH} and also, e.g.,\cite[Theorem~24.20]{BC2011}.)}
Suppose $A$ and $B$ are monotone
operators on $X$ such that $A+B$ is maximally monotone,
$\dom A \subseteq \dom B$, and $B$ is rectangular.
Then 
$\inte\ran(A+B) = \inte(\ran A + \ran B)$
and
$\overline{\ran(A+B)} = \overline{\ran A + \ran B}$.
\end{fact}

Applying the Brezis-Haraux result to our given operators $\bA$ and $\bM$,
we obtain the following. 

\begin{corollary}
\label{c:M}
The operator $\bA+\bM$ is maximally monotone and 
 $\overline{\ran(\bA+\bM)} 
= \overline{\bD^\bot + \ran\bA}$.
\end{corollary}
\begin{proof}
Since each $A_i$ is maximally monotone and recalling
Theorem~\ref{t:M}\ref{t:M1}, we see that  
$\bA$ and $\bM$ are maximally monotone.
On the other hand, $\dom\bM = \bX$.
Thus, by the well known sum theorem for maximally monotone operators (see,
e.g, \cite[Corollary~24.4(i)]{BC2011}), $\bA+\bM$ is maximally monotone.
Furthermore, 
by Theorem~\ref{t:M}\ref{t:M2}\&\ref{t:M4}, 
$\bM$ is rectangular and $\ran\bM = \bD^\perp$. 
The result therefore follows from Fact~\ref{f:B-H}. 
\end{proof}

\section{Composition}

We now use Corollary~\ref{c:M} to study the composition. 

\begin{theorem}
\label{t:main}
Suppose that $(\forall i\in I)$
$0\in\overline{\ran(\Id-T_i)}$. 
Then the following hold.
\begin{enumerate}
\item
\label{t:main1}
$\bzero\in\overline{\ran(\bA+\bM)}$. 
\item
\label{t:main2}
$(\forall\varepsilon>0)$
$(\exists (\bb,\bx)\in\bX\times \bX)$
$\|\bb\|\leq\varepsilon$
and 
$\bx=\bT(\bb+\bR\bx)$.
\item
\label{t:main3}
$(\forall\varepsilon>0)$
$(\exists (\bc,\bx)\in\bX\times \bX)$
$\|\bc\|\leq\varepsilon$
and 
$\bx=\bc+ \bT(\bR\bx)$.
\item
\label{t:main4}
$(\forall\varepsilon>0)$
$(\exists \bx\in\bX)$
$(\forall i\in I)$
$\|T_{i-1}\cdots T_1x_m-T_iT_{i-1}\cdots T_1x_m -x_{i-1}+x_i\|
\leq (2i-1)\varepsilon$,
where $x_0 = x_m$. 
\item
\label{t:main5}
$(\forall\varepsilon>0)$
$(\exists x\in X)$
$\|x-T_mT_{m-1}\cdots T_1x\|\leq m^2\varepsilon$. 
\end{enumerate}
\end{theorem}
\begin{proof}
\ref{t:main1}: 
The assumptions and \eqref{e:MintPar} imply that 
$(\forall i\in I)$ $0\in\overline{\ran A_i}$. 
Hence, $\bzero\in \overline{\ran \bA}$. 
Obviously, $\bzero\in\bD^\perp$. 
It follows that 
$\bzero\in\overline{\bD^\perp + \ran\bA}$.
Thus, by Corollary~\ref{c:M},
$\bzero\in\overline{\ran(\bA+\bM)}$. 

\ref{t:main2}: 
Fix $\varepsilon>0$. 
In view of \ref{t:main1}, there
exists $\bx\in\bX$ and $\bb\in\bX$ such that
$\|\bb\|\leq\varepsilon$ 
and $\bb\in \bA\bx+\bM\bx$. 
Hence $\bb+\bR\bx \in (\bId+\bA)\bx$ and thus
$\bx = J_{\bA}(\bb+\bR\bx) = \bT(\bb+\bR\bx)$. 

\ref{t:main3}: 
Let $\varepsilon > 0$.
By \ref{t:main2}, there exists 
$(\bb,\bx)\in\bX\times\bX)$ such that
$\|\bb\|\leq \varepsilon$ and 
$\bx = \bT(\bb+\bR\bx)$.
Set $\bc = \bx-\bT(\bR\bx)=
\bT(\bb+\bR\bx) - \bT(\bR\bx)$
Then, since $\bT$ is nonexpansive, 
$\|\bc\| = \|\bT(\bb+\bR\bx) - \bT(\bR\bx)\|
\leq \|\bb\|\leq\varepsilon$. 

\ref{t:main4}: 
Take $\varepsilon>0$.
Then, by \ref{t:main3}, there exists
$\bx\in\bX$ and $\bc\in\bX$ such that
$\|\bc\|\leq\varepsilon$
and $\bx = \bc+\bT(\bR\bx)$. 
Let $i\in I$. Then
$x_i = c_i + T_ix_{i-1}$. 
Since $\|c_i\|\leq\|\bc\|\leq\varepsilon$ and
$T_i$ is nonexpansive, we have 
\begin{subequations}
\begin{align}
\|T_iT_{i-1}\cdots T_1x_0-x_i\| 
&\leq \|T_iT_{i-1}\cdots T_1x_0-T_ix_{i-1}\|
+ \|T_ix_{i-1} - x_i\|\\
&\leq \|T_iT_{i-1}\cdots T_1x_0-T_ix_{i-1}\|
+ \varepsilon.
\end{align}
\end{subequations}
We thus obtain inductively
\begin{equation}
\label{e:trinity1}
\|T_iT_{i-1}\cdots T_1x_0-x_i\| \leq i\varepsilon.
\end{equation}
Hence,
\begin{equation}
\label{e:trinity2}
\|T_{i-1}\cdots T_1x_0-x_{i-1}\| \leq (i-1)\varepsilon.
\end{equation}
The conclusion now follows from 
adding \eqref{e:trinity1} and \eqref{e:trinity2},
and recalling the triangle inequality.

\ref{t:main5}: 
Let $\varepsilon>0$. 
In view of \ref{t:main4},
there exists $\bx \in\bX$ such that 
\begin{equation}
\label{e:sunday}
(\forall i\in I)\quad
\|T_{i-1}\cdots T_1x_m-T_iT_{i-1}\cdots T_1x_m -x_{i-1}+x_i\|
\leq (2i-1)\varepsilon
\end{equation}
where $x_0 = x_m$.
Now set $(\forall i\in I)$
$e_i = T_{i-1}\cdots T_1x_m-T_iT_{i-1}\cdots T_1x_m -x_{i-1}+x_i$.
Then $(\forall i\in I)$ $\|e_i\| \leq (2i-1)\varepsilon$. 
Set $x = x_m$. Then
\begin{align}
\sum_{i=1}^{m}e_i &=
\sum_{i=1}^{m} T_{i-1}\cdots T_1x_m-T_iT_{i-1}\cdots T_1x_m
-x_{i-1}+x_i\\
&= x-T_mT_{m-1}\cdots T_1x.
\end{align}
This, \eqref{e:sunday}, and the triangle inequality imply that 
\begin{equation}
\|x-T_mT_{m-1}\cdots T_1x\|
\leq \sum_{i=1}^{m}\|e_i\|
\leq \sum_{i=1}^{m} (2i-1)\varepsilon = m^2\varepsilon.
\end{equation}
This completes the proof.
\end{proof}

\begin{corollary}
\label{c:poppadoms}
Suppose that $(\forall i\in I)$
$0\in\overline{\ran(\Id-T_i)}$. 
Then $0\in\overline{\ran(\Id-T_mT_{m-1}\cdots T_1)}$. 
\end{corollary}
\begin{proof}
This follows from Theorem~\ref{t:main}\ref{t:main5}.
\end{proof}

\begin{remark}
The converse implication in Corollary~\ref{c:poppadoms} fails in general:
indeed, consider the case when $X\neq\{0\}$,
$m=2$, and $v\in X\smallsetminus\{0\}$.
Now set $T_1\colon X\to X\colon x\mapsto x+v$ and 
set $T_2\colon X\to X\colon x\mapsto x-v$.
Then $0\notin\overline{\ran(\Id-T_1)} = \{-v\}$ and
$0\notin\overline{\ran(\Id-T_2)} = \{v\}$; 
however, $T_2T_1 = \Id$ and $\overline{\ran(\Id-T_2T_1)}=\{0\}$.
\end{remark}

\begin{remark}
\label{r:optimal}
Corollary~\ref{c:poppadoms} is optimal in the sense that
even if $(\forall i\in I)$ we have $0 \in\ran(\Id-T_i)$,
we cannot deduce that $0\in\ran(\Id-T_mT_{m-1}\cdots T_1)$:
indeed, suppose that $X=\RR^2$ and $m=2$.
Set $C_1 := \epi \exp$ and $C_2 := \RR\times\{0\}$.
Suppose further that $T_1 = P_{C_1}$ and $T_2=P_{C_2}$.
Then $(\forall i\in I)$ $0\in\ran(\Id-T_i)$; however,
$0\in \overline{\ran(\Id-T_2T_1)}\smallsetminus\ran(\Id-T_2T_1)$. 
\end{remark}

\section{Asymptotic Regularity}

The following notions (taken from Bruck and Reich's seminal paper
\cite{BruckReich}) will be very useful to obtain stronger results.

\begin{definition}[(strong) nonexpansiveness and asymptotic regularity]
Let $S\colon X\to X$.
Then:
\begin{enumerate}
\item $S$ is \emph{nonexpansive} if
$(\forall x\in X)(\forall y\in X)$
$\|Sx-Sy\|\leq\|x-y\|$. 
\item $S$ is \emph{strongly nonexpansive} if $S$ is nonexpansive
and whenever $(x_n)_\nnn$ and $(y_n)_\nnn$ are sequences in $X$ such that
$(x_n-y_n)_\nnn$ is bounded and $\|x_n-y_n\|-\|Sx_n-Sy_n\|\to 0$,
it follows that $(x_n-y_n)-(Sx_n-Sy_n)\to 0$. 
\item $S$ is asymptotically regular if $(\forall x\in X)$ $S^nx-S^{n+1}x\to
0$. 
\end{enumerate}
\end{definition}

The next result illustrates that strongly nonexpansive mappings
generalize the notion of a firmly nonexpansive mapping. In addition, the
class of strongly nonexpansive mappings is closed under compositions. 
(In contrast, the composition of two (necessarily firmly nonexpansive) 
projectors may fail to be firmly nonexpansive.)

\begin{fact}[Bruck and Reich]
\label{f:B-R}
The following hold.
\begin{enumerate}
\item
\label{f:B-R0}
Every firmly nonexpansive mapping is strongly nonexpansive.
\item
\label{f:B-R1}
The composition of finitely many strongly nonexpansive mappings 
is also strongly nonexpansive. 
\end{enumerate}
\end{fact}
\begin{proof}
\ref{f:B-R0}:
See \cite[Proposition~2.1]{BruckReich}. 
\ref{f:B-R1}:
See \cite[Proposition~1.1]{BruckReich}. 
\end{proof}

The sequences of iterates and of differences of iterates have 
striking convergence properties as we shall see now. 

\begin{fact}[Bruck and Reich]
\label{f:Bru}
Let $S\colon X\to X$ be strongly nonexpansive and let $x\in X$. 
Then the following hold.
\begin{enumerate}
\item
\label{f:Bru2}
The sequence $(S^nx-S^{n+1}x)_\nnn$ converges strongly to the unique
element of least norm in $\overline{\ran(\Id-S)}$. 
\item
\label{f:Bru3}
If $\Fix S=\varnothing$, then $\|S^nx\|\to\pinf$.
\item
\label{f:Bru4}
If $\Fix S\neq\varnothing$, then
$(S^nx)_\nnn$ converges weakly to a fixed point of $S$. 
\end{enumerate}
\end{fact}
\begin{proof}
\ref{f:Bru2}:
See \cite[Corollary~1.5]{BruckReich}. 
\ref{f:Bru3}:
See \cite[Corollary~1.4]{BruckReich}. 
\ref{f:Bru4}:
See \cite[Corollary~1.3]{BruckReich}. 
\end{proof}

Suppose $S\colon X\to X$ is asymptotically regular.
Then, for every $x\in X$, 
$0\leftarrow S^nx-S^{n+1}x = (\Id-S)S^nx \in\ran(\Id-S)$
and hence $0\in\overline{\ran(\Id-S)}$. 
The opposite implication fails in general (consider $S=-\Id$), but it is
true for strongly nonexpansive mappings.

\begin{corollary}
\label{c:pizza}
Let $S\colon X\to X$ be strongly nonexpansive.
Then $S$ is asymptotically regular if and only 
if $0\in\overline{\ran(\Id-S)}$.
\end{corollary}
\begin{proof}
``$\Rightarrow$'': Clear.
``$\Leftarrow$'': Fact~\ref{f:Bru}\ref{f:Bru2}.
\end{proof}

\begin{corollary}
\label{c:B-R} 
Set $S := T_mT_{m-1}\cdots T_1$.
Then $S$ is asymptotically regular
if and only if 
$0\in\overline{\ran(\Id-S)}$.
\end{corollary}
\begin{proof}
Since each $T_i$ is firmly nonexpansive, it is also strongly
nonexpansive by Fact~\ref{f:B-R}\ref{f:B-R0}. 
By Fact~\ref{f:B-R}\ref{f:B-R1}, 
$S$ is strongly nonexpansive. 
Now apply Corollary~\ref{c:pizza}. 
Alternatively, $0\in\overline{\ran(\Id-S)}$ by Corollary~\ref{c:poppadoms}
and again Corollary~\ref{c:pizza} applies.
\end{proof}

We are now ready for our first main result. 

\begin{theorem}
\label{t:Main}
Suppose that each $T_i$ is asymptotically regular.
Then $T_mT_{m-1}\cdots T_1$ is asymptotically regular as well. 
\end{theorem}
\begin{proof}
Theorem~\ref{t:main}\ref{t:main5} implies that
$0\in\overline{\ran(\Id-T_mT_{m-1}\cdots T_1)}$. 
The conclusion thus follows from Corollary~\ref{c:B-R}. 
\end{proof}

As an application of Theorem~\ref{t:Main}, we obtain the main result
of \cite{B2003}.

\begin{example}
Let $C_1,\ldots,C_m$ be nonempty closed convex subsets of $X$.
Then the composition of the corresponding projectors,
$P_{C_m}P_{C_{m-1}}\cdots P_{C_1}$ is asymptotically regular.
\end{example}
\begin{proof}
For every $i\in I$, the projector $P_{C_i}$ is firmly nonexpansive,
hence strongly nonexpansive, and $\Fix P_{C_i}=C_i\neq\varnothing$.
Suppose that $(\forall i\in I)$ $T_i = P_{C_i}$, which is thus
asymptotically regular by Corollary~\ref{c:pizza}.
Now apply Theorem~\ref{t:Main}. 
\end{proof}

\section{Convex Combination}

In this section, 
we use our fixed weights $(\lambda_i)_{i\in I}$ (see
\eqref{e:weights}) to turn $X^m$ into a Hilbert product
space \emph{different from} $\bX$ considered in the previous sections.
Specifically, we set
\boxedeqn{
\bY := X^m
\quad\text{with}\quad
\scal{\bx}{\by} = \sum_{i\in I}\lambda_i\scal{x_i}{y_i}
}
so that $\|\bx\|^2 = \sum_{i\in I} \lambda_i \|x_i\|^2$.
We also set 
\boxedeqn{
\bQ\colon X^m \to X^m \colon
 \bx \mapsto (\bar{x})_{i\in I},
\quad\text{where $\bar{x} := \sum_{i\in I}\lambda_ix_i$.}
}

\begin{fact}
\label{f:Q}
{\rm (See \cite[Proposition~28.13]{BC2011}.)}
In the Hilbert product space $\bY$, we have 
$P_{\bD} = \bQ$.
\end{fact}

\begin{corollary}
\label{c:Q}
In the Hilbert product space $\bY$, 
the operator $\bQ$ is firmly nonexpansive and strongly
nonexpansive.
Furthermore, $\Fix\bQ=\bD\neq\varnothing$,
$\bzero\in\ran(\bId-\bQ)$, and $\bQ$ is asymptotically regular. 
\end{corollary}
\begin{proof}
By Fact~\ref{f:Q}, the operator $\bQ$ is equal to the 
projector $P_\bD$ and
hence firmly nonexpansive. Now apply Fact~\ref{f:B-R}\ref{f:B-R0}
to deduce that $\bQ$ is strongly nonexpansive.
It is clear that
$\Fix\bQ=\bD$ and that $\bzero\in\ran(\bId-\bQ)$. 
Finally, recall Corollary~\ref{c:pizza} to see that $\bQ$ is
asymptotically regular. 
\end{proof}

\begin{proposition}
\label{p:T}
In the Hilbert product space $\bY$, the operator $\bT$ is firmly
nonexpansive. 
\end{proposition}
\begin{proof}
Since each $T_i$ is firmly nonexpansive,
we have $(\forall \bx=(x_i)_{i\in I}\in\bY)
(\forall \by=(y_i)_{i\in I}\in\bY)$
$\|T_ix_i-T_iy_i\|^2 \leq \scal{x_i-y_i}{T_ix_i-T_iy_i}$
$\Rightarrow$ 
$\|\bT\bx-\bT\by\|^2 = 
\sum_{i\in I}\lambda_i \|T_ix_i-T_iy_i\|^2 \leq \sum_{i\in
I}\lambda_i\scal{x_i-y_i}{T_ix_i-T_iy_i} =
\scal{\bx-\by}{\bT\bx-\bT\by}$.
\end{proof}

\begin{theorem}
\label{t:gurke}
Suppose that $(\forall i\in I)$
$0\in\overline{\ran(\Id-T_i)}$. 
Then the following hold in the Hilbert product space $\bY$.
\begin{enumerate}
\item 
\label{t:gurke1}
$\bzero \in \overline{\ran(\bId-\bT)}$
\item 
\label{t:gurke2}
$\bT$ is asymptotically regular. 
\item 
\label{t:gurke3}
$\bQ\circ\bT$ is asymptotically regular. 
\end{enumerate}
\end{theorem}
\begin{proof}
\ref{t:gurke1}:
This follows because
$(\forall\bx=(x_i)_{i\in I})$
$\|\bx-\bT\bx\|^2 = \sum_{i\in I}\lambda_i\|x_i-T_ix_i\|^2$.

\ref{t:gurke2}:
Combine Fact~\ref{f:B-R}\ref{f:B-R0} with Corollary~\ref{c:pizza}. 

\ref{t:gurke3}:
On the one hand, 
$\bQ$ is firmly nonexpansive and asymptotically regular
by Corollary~\ref{c:Q}.
On the other hand,
$\bT$ is firmly nonexpansive and asymptotically regular
by Proposition~\ref{p:T} and \ref{t:gurke2}. 
Altogether, the result follows from
Theorem~\ref{t:Main}. 
\end{proof}

We are now ready for our second main result.

\begin{theorem}
\label{t:haupt}
Suppose that each $T_i$ is asymptotically regular. 
Then $\sum_{i\in I}\lambda_iT_i$ is asymptotically regular as well.
\end{theorem}
\begin{proof}
Set $S:=\sum_{i\in I}\lambda_iT_i$.
Fix $x_0\in X$ and set $(\forall\nnn)$
$x_{n+1} = Sx_n$.
Set $\bx_0 = (x_0)_{i\in I}\in X^m$
and $(\forall\nnn)$  $\bx_{n+1}=(\bQ\circ\bT)\bx_n$. 
Then $(\forall\nnn)$ $\bx_n = (x_n)_{i\in I}$. 
Now $\bQ\circ\bT$ is asymptotically regular by
Theorem~\ref{t:gurke}\ref{t:gurke3}; hence,
$\bx_n-\bx_{n+1}=(x_n-x_{n+1})_{i\in I}\to \bzero$. 
Thus, $x_n-x_{n+1}\to 0$ and therefore $S$ is asymptotically regular. 
\end{proof}

\begin{remark}
Theorem~\ref{t:haupt} extends 
\cite[Theorem~4.11]{BMWnear} from
Euclidean to Hilbert space.
One may also prove Theorem~\ref{t:haupt} along
the lines of the paper \cite{BMWnear}; however, 
that route takes longer. 
\end{remark}

\begin{remark}
Similarly to Remark~\ref{r:optimal}, one cannot deduce that
if each $T_i$ has fixed points, then $\sum_{i\in I}\lambda_i T_i$
has fixed points as well: indeed, consider the setting described in
Remark~\ref{r:optimal} for an example. 
\end{remark}

We conclude this paper by showing that
we truly had to work in $\bY$ and not in $\bX$;
indeed, viewed in $\bX$, the operator $\bQ$ is generally
not even nonexpansive.

\begin{theorem}
\label{t:willow}
Suppose that $X\neq\{0\}$. 
Then the following are equivalent 
in the Hilbert product space $\bX$. 
\begin{enumerate}
\item
\label{t:willow1}
$(\forall i\in I)$ $\lambda_i = 1/m$.
\item
\label{t:willow2}
$\bQ$ coincides with the projector $P_{\bD}$. 
\item
\label{t:willow3}
$\bQ$ is firmly nonexpansive.
\item
\label{t:willow4}
$\bQ$ is nonexpansive.
\end{enumerate}
\end{theorem}
\begin{proof}
``\ref{t:willow1}$\Rightarrow$\ref{t:willow2}'': 
\cite[Proposition~25.4(iii)]{BC2011}. 
``\ref{t:willow2}$\Rightarrow$\ref{t:willow3}'': 
Clear. 
``\ref{t:willow3}$\Rightarrow$\ref{t:willow4}'': 
Clear. 
``\ref{t:willow4}$\Rightarrow$\ref{t:willow1}'': 
Take $e\in X$ such that $\|e\|=1$.
Set $\bx := (\lambda_i e)_{i\in I}$ and 
$y := \sum_{i\in I}\lambda_i^2e$. 
Then $\bQ\bx = (y)_{i\in I}$. 
We compute
$\|\bQ\bx\|^2 = m\|y\|^2 = m\big(\sum_{i\in I}\lambda_i^2\big)^2$ 
and 
$\|\bx\|^2 = \sum_{i\in I}\lambda_i^2$. 
Since $\bQ$ is nonexpansive,
we must have that $\|\bQ\bx\|^2\leq \|\bx\|^2$, which is equivalent to 
\begin{equation}
m\Big(\sum_{i\in I}\lambda_i^2\Big)^2 \leq \sum_{i\in I}\lambda_i^2
\end{equation}
and to 
\begin{equation}
\label{e:1211a}
m\sum_{i\in I}\lambda_i^2 \leq 1.
\end{equation}
On the other hand, applying the Cauchy-Schwarz inequality
to the vectors $(\lambda_i)_{i\in I}$ and $(1)_{i\in I}$ in $\RR^m$
yields
\begin{equation}
\label{e:1211b}
1 = 1^2 = \Big(\sum_{i\in I}\lambda_i\cdot 1\Big)^2
\leq \big\|(\lambda_i)_{i\in I}\big\|^2
\big\|(1)_{i\in I}\big\|^2 = m\sum_{i\in I}\lambda_i^2. 
\end{equation}
In view of \eqref{e:1211a}, the Cauchy-Schwarz inequality \eqref{e:1211b}
is actually an equality which implies that $(\lambda_i)_{i\in I}$ is
a multiple of $(1)_{i\in I}$. We deduce that $(\forall i\in I)$
$\lambda_i = 1/m$. 
\end{proof}

\section*{Acknowledgments}
Part of this research was initiated during
a research visit of VMM at the Kelowna campus of UBC in Fall 2009. 
HHB was partially supported by the Natural Sciences and
Engineering Research Council of Canada and by the Canada Research Chair
Program.
VMM was partially supported by 
DGES, Grant BFM2009-1096-C02-01 and Junta de Andalucia, Grant FQM-127.
SMM was partially supported 
by the Natural Sciences and Engineering Research Council
of Canada.
XW was partially supported 
by the Natural Sciences and Engineering Research Council
of Canada.

\end{document}